\documentclass{article}
\usepackage{mystyle}
\usepackage[cp1251]{inputenc}
\usepackage{amssymb}
\usepackage{amsmath}
\usepackage{latexsym}
\usepackage{amsthm}
\newtheorem {theorem}{Theorem}
\newtheorem {definition}{Definition}
\newtheorem {proposition}{Proposition}

\newtheorem*{MT}{Main Theorem}
\theoremstyle{remark}
\newtheorem*{remark}{Remark}

\DeclareMathOperator{\vol}{Vol}
\DeclareMathOperator{\area}{Area}
\DeclareMathOperator{\im}{Im}

\begin{document}
\title{Spectral properties of a family of minimal tori \\of revolution in five-dimensional sphere}
\author{Mikhail A. Karpukhin}
\date{} 
\maketitle
\begin{abstract}
The normalized eigenvalues $\Lambda_i(M,g)$ of the Laplace-Beltrami operator can be considered as functionals on the space of all Riemannian metrics $g$ on a fixed surface $M$. In recent papers several explicit examples of extremal metrics were provided. These metrics are induced by minimal immersions of surfaces in $\mathbb{S}^3$ or $\mathbb{S}^4$. In the present paper a family of extremal metrics induced by minimal immersions in $\mathbb{S}^5$ is provided.  
\end{abstract} 
\section*{Introduction}
Let $M$ be a closed surface and $g$ be a Riemannian metric on $M$. Then the Laplace-Beltrami operator $\Delta\colon C^\infty(M)\to C^\infty(M)$ is given by the formula
$$
\Delta f = -\frac{1}{\sqrt{|g|}}\frac{\partial}{\partial x^i}\bigl(\sqrt{|g|}g^{ij}\frac{\partial f}{\partial x^j}\bigr).
$$
The spectrum of $\Delta$ consists only of eigenvalues. Let us denote them by 
$$
0 = \lambda_0(M,g) < \lambda_1(M,g) \leqslant \lambda_2(M,g) \leqslant \lambda_3(M,g) \leqslant \ldots,
$$
where the eigenvalues are written with their multiplicities.

In this paper the family of functionals
$$
\Lambda_i(M,g) = \lambda_i(M,g)\area (M,g)
$$
is investigated.
Let us fix $M$. We are interested in the quantity $\sup \Lambda_i(M,g)$, where the supremum is taken over the space of all Riemannian metrics on $M$.

An upper bound for $\Lambda_1(M,g)$ in terms of the genus of $M$ was provided in the paper~\cite{YangYau} and later the existence of an upper bound for $\Lambda_i(M,g)$ was shown in the paper~\cite{Korevaar}.

Several recent papers~\cite{EGJ, ElSoufiIlias1, ElSoufiIlias2, Hersh, JNP, LiYau, Nadirashvili1, Nadirashvili2} deal with finding the exact values of this supremum in the space of all Riemannian metrics on several particular surfaces. We refer to the introduction to the paper~\cite{PenskoiOtsuki} for more details.

In attempt to solve this problem,  the following definition was introduced in several papers, see e.g.~\cite{ElSoufiIlias1,Nadirashvili1}.
\begin{definition} A Riemannian metric $g$ on a closed surface $M$ is called an extremal metric for a functional 
$\Lambda_i(M,g)$ if for any analytic deformation $g_t$ such that $g_0 = g$ the following inequality holds,
$$
\frac{d}{dt}\Lambda_i(M,g_t)\Bigl|_{t=0+} \leqslant 0 \leqslant \frac{d}{dt}\Lambda_i(M,g_t)\Bigl|_{t=0-}.
$$
\end{definition}
For the correctness of this definition we refer to the papers~\cite{BU,Berger,ElSoufiIlias2}. 

A real breakthrough in finding explicit examples of (smooth) extremal metrics became possible due to connection with the theory of minimal surfaces in spheres discovered in the paper~\cite{ElSoufiIlias2}.
Let $\psi\colon  M \looparrowright \mathbb{S}^n$ be a minimal immersion in the unit sphere. We denote by $\Delta$ the Laplace-Beltrami operator on $M$ associated with the metric induced by the immersion $\psi$. Let us introduce the Weyl's counting funcion
$$
N(\lambda) = \#\{i|\lambda_i(M,g)<\lambda\}.
$$
The following theorem provides a general approach to find smooth extremal metrics.
\begin{theorem}[El Soufi, Ilias~\cite{ElSoufiIlias2}] Let $\psi\colon M \looparrowright 
\mathbb{S}^n$ be a minimal
immersion of a surface in the unit sphere $\mathbb{S}^n$ endowed with the canonical metric $g_{can}$. Then the metric $\psi^*g_{can}$ on $M$ is extremal for the functional $\Lambda_{N(2)}(M,g)$.
\label{th1}
\end{theorem}

In recent papers~\cite{Karpukhin2, Lapointe,PenskoiLawson, PenskoiOtsuki} this connection was used to provide several examples of extremal metrics on the torus and the Klein bottle. These metrics were induced on these surfaces by minimal immersions in $\mathbb{S}^3$ and $\mathbb{S}^4$. In the present paper a family of minimally immersed surfaces in $\mathbb{S}^5$ is investigated. For any pair of positive integers $m,n$ such that $m\geqslant n$ and $(m,n) = 1$ we consider a doubly $2\pi$-periodic immersion 
$\varphi_{m,n}\colon \mathbb{R}^2 \to \mathbb{S}^5$, given by the formula   
\begin{equation}
\begin{split}
&\varphi_{m,n}(x,y) =\\ 
&\left(\sqrt{\frac{m+n}{2m+n}}e^{imy}\sin x, \sqrt{\frac{m+n}{m+2n}}e^{iny}\cos x,\sqrt{\frac{n\cos^2x}{m+2n} + \frac{m\sin^2x}{2m+n}}e^{-i(m+n)y}\right),
\end{split}
\label{phi}
\end{equation}
where $\mathbb{S}^5$ is considered as the set of unit length vectors in $\mathbb{C}^3$. We denote the image of $\varphi_{m,n}$ by $M_{m,n}$. To the best of author's knowledge, explicit formula~(\ref{phi}) first appeared in the introduction to the paper~\cite{Mironov}. This immersion can be obtained due to a general construction by Mironov (see the paper~\cite{Mironov2}). We should mention that $M_{m,n}$ were described in the conformal coordinates in the papers~\cite{Haskins, Joyce}. The main result of the present paper is the following theorem.
\begin{MT}
\label{MainTheorem}
For any pair of positive integers $m,n$ such that $m\geqslant n$ and $(m,n) = 1$ the immersion $\varphi_{m,n}$ is minimal. The corresponding surface $M_{m,n}$ is the  torus. If $mn\equiv 0 \mod 2$ then the metric induced on $M_{m,n}$ by the immersion is extremal for the functional $\Lambda_{4(m+n)-3}(\mathbb{T}^2,g)$. If $mn\equiv 1 \mod 2$ then the metric induced on $M_{m,n}$ by the immersion is extremal for the functional $\Lambda_{2(m+n)-3}(\mathbb{T}^2,g)$.
\end{MT}
The proof of this theorem is similar to the proof of the main theorem in the paper~\cite{PenskoiLawson} by Penskoi. Although, we should mention that the exposition here is much simplified, e.g. we do not use the theory of the Magnus-Winkler-Ince equation as opposed to~\cite{PenskoiLawson}. We also fill a gap by giving a rigorous proof of Proposition~20 from the paper~\cite{PenskoiLawson}. 

We provide the exact value of the corresponding functional in terms of elliptic integrals of the first and the second kind given respectively by formulae
$$
K(k) = \int\limits_0^1\frac{1}{\sqrt{1-x^2}\sqrt{1-k^2x^2}}\,dx,\qquad
E(k) = \int\limits_0^1\frac{\sqrt{1-k^2x^2}}{\sqrt{1-x^2}}\,dx.
$$ 

Following the paper~\cite{Karpukhin2} we also prove the non-maximality of the metric on $M_{m,n}$.
\begin{proposition}
If $mn\equiv 0 \mod 2$ then 
\begin{equation*}
\begin{split}
& \Lambda_{4(m+n)-3}(M_{m,n}) =\\
& 16\pi\left(\sqrt{m^2+2mn}\,E\left(\sqrt{\frac{m^2-n^2}{m^2+2mn}}\right) - \frac{mn}{\sqrt{m^2+2mn}}K\left(\sqrt{\frac{m^2-n^2}{m^2+2mn}}\right)\right).
\end{split}
\end{equation*}
If $mn\equiv 1 \mod 2$ then
\begin{equation*}
\begin{split}
& \Lambda_{2(m+n)-3}(M_{m,n}) = \\
& 8\pi\left(\sqrt{m^2+2mn}\,E\left(\sqrt{\frac{m^2-n^2}{m^2+2mn}}\right) - \frac{mn}{\sqrt{m^2+2mn}}K\left(\sqrt{\frac{m^2-n^2}{m^2+2mn}}\right)\right).
\end{split}
\end{equation*}
For every pair $\{m,n\}\ne\{1,1\}$ the metric on $M_{m,n}$ is not maximal for the corresponding functional.
\label{MainProposition}
\end{proposition}
\begin{remark}
It is easy to check that $\varphi_{1,1}$ is an immersion of the flat equilateral torus in $\mathbb{S}^5$ by first eigenfuctions and as it was shown in the paper~\cite{Nadirashvili1} this metric is maximal for the functional $\Lambda_1(\mathbb{T}^2,g)$.
\end{remark}
The paper is organized in the following way. In Section~\ref{min} we describe $M_{m,n}$ as a part of a general construction from the paper~\cite{Mironov2} by Mironov. Then in Section~\ref{StL} we reduce the problem of finding $N(2)$ for $\Delta$ to the similar problem for a family of periodic Sturm-Liouville operators.
Finally, Section~\ref{proof} contains the proof of Main Theorem and Section~\ref{proof1} is dedicated to the proof of Proposition~\ref{MainProposition}.
\section{Preliminaries}
\label{prelim}
\subsection{Construction of minimal Lagrangian submanifolds in $\mathbb{C}^n$ by Mironov}
Let $M$ be a $k$-dimensional submanifold of $\mathbb{R}^n$ given by equations
$$
e_{1j}u_1^2 + \ldots + e_{nj}u_n^2 = d_j, \qquad j=1,\ldots,n-k,
$$
where $d_j\in \mathbb{R}$ and $e_{ij}\in \mathbb{Z}$. Since $\dim M = k$, the vectors $e_j = (e_{j1},\ldots,e_{j(n-k)})\in \mathbb{Z}^{n-k}$, $j = 1,\ldots,n$ form a lattice $\Lambda$ of maximal rank in $\mathbb{R}^{n-k}$. Let us denote by $\Lambda^*$ the dual lattice to $\Lambda$
$$
\Lambda^* = \{y\in\mathbb{R}^{n-k}|(e_i,y)\in \mathbb{Z}, i=1,\ldots,n\},
$$
where $(x,y) = x_1y_1+\ldots+x_{n-k}y_{n-k}$.

Consider the map $\varphi\colon M\times (\mathbb{R}^{n-k}/\Lambda^*)\to \mathbb{C}^n$ given by the explicit formula
$$
\varphi(u_1,\ldots,u_n,y) = (u_1e^{2\pi i (e_1,y)},\ldots,u_ne^{2\pi i (e_n,y)}).
$$  
We endow $\mathbb{C}^n$ with the standard symplectic form
$$
\omega = dx^1\wedge dy^1 + \ldots + dx^n\wedge dy^n.
$$
Recall that an immersion $\psi\colon N\looparrowright\mathbb{C}^n$ is called Lagrangian if $\psi^*\omega = 0$.
\label{min}
\begin{theorem}[Mironov~\cite{Mironov2}]
\label{MironovTheorem}
Suppose that $e_1 + \ldots + e_n=0$. Then the immersion $\varphi$ is a minimal Lagrangian immersion.
\end{theorem}

Let us now consider a particular case $M = \{(x_1,x_2,x_3)|mx_1^2+nx_2^2 - (m+n)x_3^2 = 0\}\subset \mathbb{R}^3$. Then by Theorem~\ref{MironovTheorem}, the immersion $\varphi$ is a minimal Lagrangian immersion.  It is easy to see that in this case $\im\varphi$ is a cone $C(M_{m,n})$ over $M_{m,n}$. It is a standard fact that $C(M_{m,n})$ is minimal in $\mathbb{C}^3$ iff $M_{m,n}$ is minimal in $\mathbb{S}^5\subset \mathbb{C}^3$, see e.g. the paper~\cite{Simons}. 


\subsection{Symmetries of $\varphi_{m,n}$.}
The goal of this section is to prove the following proposition.
\begin{proposition}
If $mn\equiv 1\mod 2$ then one has $\varphi_{m,n}(x,y) = \varphi_{m,n}(x+\pi,y+\pi)$ and $\varphi_{m,n}|_{[0,2\pi)\times [0,2\pi)}$ is a double cover almost everywhere. 
If $mn\equiv 0\mod 2$ then $\varphi_{m,n}|_{[0,2\pi)\times [0,2\pi)}$ is one-to-one almost everywhere. Thus $M_{m,n}$ is a torus for each $m,n>0$, $(m,n)=1$.
\end{proposition}
\begin{remark}
In fact, according to the paper~\cite{Mironov2}, one can omit the words "almost everywhere" in the previous proposition.
\end{remark}
\begin{proof}
Since $(m,n) = 1$, there are no symmetries of the form $(x,y)\mapsto(x,y+\alpha)$. Examining the third coordinate of $\varphi_{m,n}$, we see that the only possible symmetry has the form $(x,y)\mapsto \left(x+\pi,y+\frac{2\pi}{m+n}\right)$. Substituting this into the first two coordinates of $\varphi_{m,n}$ we obtain the statement of the proposition.
\end{proof}

\subsection{Associated periodic Sturm-Liouville problem.}
In this section we reduce the problem of finding $N(2)$ for the Laplace-Beltrami operator on $M_{m,n}$ to a similar problem for the associated Sturm-Liouville operator. 
\label{StL}

Direct calculations show that the metric on $M_{m,n}$ is given by 
$$
g = (m+n)\left(\frac{(m+n)-(m-n)\cos 2x}{m^2+4mn+n^2 - (m^2-n^2)\cos 2x} dx^2 + \frac{1}{2}(m+n-(m-n)\cos 2x)dy^2\right).
$$

Let us introduce the notations $\sigma(x) = \sqrt{m^2+4mn+n^2 - (m^2-n^2)\cos 2x}$ and $\rho(x)= (m+n)(m+n-(m-n)\cos 2x)$. Then a straightforward calculation shows that
the following formula holds for the Laplace-Beltrami operator,
\begin{equation}
\label{Laplace}
\Delta f = - \frac{1}{\rho(x)}\left(\sigma(x)\frac{\partial}{\partial x}\left(\sigma(x)\frac{\partial f}{\partial x}\right) + 2 \frac{\partial^2 f}{\partial y^2}\right).
\end{equation}

\begin{proposition}
Assume $mn\equiv 0\mod 2$. The number $\lambda$ is the eigenvalue of the Laplace-Beltrami operator~(\ref{Laplace}) if and only if there exists $l\in\mathbb{Z}_{\geqslant 0}$ such that there is a solution of the following associated periodic Sturm-Liouville problem
\begin{equation}
\label{assSL}
\begin{split}
&-\sigma(x)\frac{d}{dx}\left(\sigma(x)\frac{d h(x)}{dx}\right) + 2l^2 h(x) = \lambda\rho(x) h(x), \\
&h(x+2\pi) \equiv h(x).
\end{split}
\end{equation}
Corresponding eigenspace is spanned by the functions of the form $h(l,x)\sin lx$ and $h(l,x)\cos lx$, where $l$ is any positive integer number such that a solution of problem~(\ref{assSL}) exists and $h(l,x)$ is the corresponding solution.

If $mn\equiv 1 \mod 2$ then the statement remains the same with the boundary conditions
\begin{equation}
\label{bc}
h(x+\pi) \equiv (-1)^lh(x).
\end{equation}
\end{proposition}
\begin{proof}
Let us remark that $\Delta$ commutes with $\dfrac{\partial^2}{\partial y^2}$. Thus these operators have a common basis of eigenfunctions of the form $h(l,x)\cos lx$ and $h(l,x)\sin lx$. By substituting these eigenfunctions into formula~(\ref{Laplace}) we obtain equation~(\ref{assSL}). Since any function on $M_{m,n}$ should be doubly $2\pi$-periodic, we have $l\in \mathbb{Z}_{\geqslant 0}$ and the boundary conditions in~(\ref{assSL}).

In the case $mn\equiv 1 \mod 2$ any function $f\in C^{\infty}(M_{m,n})$ should satisfy the condition $f(x+\pi,y+\pi) = f(x,y)$. This condition implies immediately boundary conditions~(\ref{bc}).
\end{proof}

For a general Sturm-Liouville problems the following classic proposition holds, see e.g. book~\cite{CL}.
\begin{proposition} Consider a periodic Sturm-Liouville problem in the form 
\begin{equation}
\label{SL}
\begin{split}
& -\frac{d}{dt}\left(p(t)\frac{d}{dt}h(t)\right) + q(t)h(t) = \lambda r(t) h(t),\\
& h(t+t_0)\equiv h(t),
\end{split}
\end{equation}
where $p(t),r(t)>0$ and $p(t+t_0)\equiv p(t),q(t+t_0)\equiv q(t),r(t+t_0)\equiv r(t)$. Let us denote by 
$\lambda_i$ and $h_i(t)$ ($i = 0,1,2,\ldots$) the eigenvalues and eigenfunctions of the problem (\ref{SL}).
Then the following inequalities hold, 
$$
\lambda_0<\lambda_1\leqslant\lambda_2<\lambda_3\leqslant\lambda_4<\lambda_5\leqslant\lambda_6\ldots
$$
For $\lambda = \lambda_0$ there exists a one-dimensional eigenspace spanned by $h_0(t)$.
For $i\geqslant 0$ if $\lambda_{2i+1}<\lambda_{2i+2}$ then there is a one-dimensional $\lambda_{2i+1}$-eigenspace spanned by $h_{2i+1}(t)$ 
 and there is a one-dimensional $\lambda_{2i+2}$-eigenspace spanned by $h_{2i+2}(t)$. If $\lambda_{2i+1} = \lambda_{2i+2}$ then there is a two-dimensional eigenspace spaned by
$h_{2i+1}(t)$ and $h_{2i+2}(t)$ with eigenvalue $\lambda = \lambda_{2i+1} = \lambda_{2i+2}$. \label{SturmProp}

The eigenfunction $h_0(t)$ has no zeros on $[0,t_0)$. The eigenfunctions $h_{2i+1}(t)$ 
and $h_{2i+2}(t)$ each have
exactly $2i+2$ zeros on $[0,t_0)$. 
\end{proposition}
\begin{proposition}
For $l\geqslant 0$ the eigenvalues $\lambda_i(l)$ of problem~(\ref{assSL}) are strictly increasing functions of the parameter $l$.
\label{monotonicity} 
\end{proposition}
\begin{proof}
The Raleigh quotient for equation~(\ref{assSL}) is defined by the formula
$$
R_l[f] = \frac{\displaystyle\int\limits_0^{2\pi} \left(\sigma(x)(f')^2 + \frac{2l^2}{\sigma(x)}f^2\right)\,dx}{\displaystyle\int\limits_0^{2\pi}\frac{\rho(x)}{\sigma(x)}f^2\,dx}.
$$
By variational characterization of eigenvalues (see e.g. the book~\cite{Henrot}), one has
$$
\lambda_k(l) = \inf_{E_k}\sup_{f\in E_k} R_l[f],
$$
where infimum is taken over all $(k+1)$-dimensional subspaces in the space of all $2\pi$-periodic functions of the Sobolev space $H^1[0,2\pi]$. Moreover, the infimum is reached on the space $V_k(l)$ formed by the first $(k+1)$ eigenfunctions. Let us note that $R_{l_1}[f]<R_{l_2}[f]$ if $0\leqslant l_1 < l_2$.

Then $\lambda_k(l_1)\leqslant \sup_{f\in V_k(l_2)}R_{l_1}[f]$. The latter supremum is reached on some function $g\in V_k(l_2)$. Thus one has
$$
\lambda_k(l_1)\leqslant R_{l_1}[g]< R_{l_2}[g]\leqslant \sup_{f\in V_k(l_2)} = \lambda_k(l_2),
$$
which completes the proof.
\end{proof}

\section{Proof of Main Theorem}
\label{proof}
\subsection{Proof of the Main Theorem.}
We need the following classic theorem (see e.g. the book~\cite{KN}).
\begin{theorem} Let $M \looparrowright \mathbb{S}^n$ be a minimally immersed submanifold of the unit sphere 
$\mathbb{S}^n \subset \mathbb{R}^{n+1}$. Then the restrictions $x^1|_M,\ldots,x^{n+1}|_M$ on $M$ of the standard coordinate
functions of $\mathbb{R}^{n+1}$ are eigenfunctions of the Laplace-Beltrami operator on $M$ with eigenvalue $\dim M$.
\label{th2}
\end{theorem}

According to Theorem~\ref{th2}, the components of $\varphi_{m,n}$ are eigenfunctions of the Laplace-Beltrami operator on $M_{m,n}$. Since, the function  
$$
\sqrt{\frac{n\cos^2x}{m+2n} + \frac{m\sin^2x}{2m+n}}
$$ 
does not have zeroes on $[0,2\pi)$, we have
$$
h_0(m+n,x) = \sqrt{\frac{n\cos^2x}{m+2n} + \frac{m\sin^2x}{2m+n}}
$$
and $\lambda_0(m+n)=2$. By Proposition~\ref{monotonicity} one has $\lambda_0(l)<2$ for $l<m+n$. Similarly both $\sin x$ and $\cos x$ have $2$ zeroes on $[0,2\pi)$. Thus, again by Proposition~\ref{monotonicity}, either $\lambda_1(m) = 2$ and $\lambda_2(n) = 2$ or $\lambda_1(n) = 2$ and $\lambda_2(m) = 2$. In the latter case we have a contradiction since $m>n$ and $2=\lambda_1(n)<\lambda_1(m)\leqslant\lambda_2(m) = 2$. Thus, $\lambda_1(l)<2$ for $l<m$ and $\lambda_2(l)<2$ for $l<n$.
The last part of the proof is based on the following proposition which we prove later this section.
\begin{proposition}
The eigenvalue $\lambda_3(l)$ of problem~(\ref{assSL}) satisfies the inequality $\lambda_3(0)>2$.
\label{lambda3}
\end{proposition}
Recall that for every $\lambda_i(l)$ with $l>0$ there are two eigenfuctions of the Laplace-Beltrami operator on $M_{m,n}$. This observation completes the proof in the case $mn\equiv 0 \mod 2$. 

If $mn\equiv 1\mod 2$ then one has to take into account the symmetry $(x,y)\mapsto (x+\pi,y+\pi)$, i.e. if $l$ is even then we need to count only $\pi$-periodic solutions of equation~(\ref{assSL}) and if $l$ is odd then we need to count only $\pi$-antiperiodic solutions of~(\ref{assSL}). Application of Proposition~\ref{SturmProp} with $t_0 = \pi,2\pi$ yields the fact that $h_{2i+1}$ and $h_{2i+2}$ are $\pi$-antiperiodic iff $i$ is odd and $\pi$-periodic otherwise. Obvious calculations complete the proof. 

The rest of this section is dedicated to the proof of Proposition~\ref{lambda3}.

\subsection{Lam\'e equation.}
In this section we recall several facts concerning the Lam\'e equation usually written as 
\begin{equation}
\label{Lame}
\frac{d^2\varphi}{dz^2} + (h-n(n+1)k^2\mathrm{sn}^2 z)\varphi = 0.
\end{equation}

We use a trigonometric form of the Lam\'e equation
\begin{equation}
\label{TrigLame}
[1-(k\cos x)^2]\frac{d^2\varphi}{dx^2} + k^2\sin y\cos y\frac{d\varphi}{dy} + [h - n(n+1)(k\cos y)^2]\varphi = 0. 
\end{equation}
Equation~(\ref{TrigLame}) can be obtained from equation~(\ref{Lame}) using the following change of variables
$$
\mathrm{sn} z = \cos y\quad \Leftrightarrow \quad y = \frac{\pi}{2} - \mathrm{am} z,
$$ 
where $\mathrm{am}$ is the Jacobi amplitude function, see e.g. the book~\cite{EMOT}. 

In order to prove Proposition~\ref{lambda3} we need the following proposition. 

\begin{proposition}
Assume $n=1$. Then the eigenvalue $h_3(k)$ is greater than $2$ for every $0<k<1$. 
\end{proposition}
\begin{proof}
According to the paper~\cite{Volkmer} the number $h_3(k)$ can be characterized as the first eigenvalue of the problem~(\ref{TrigLame}) with boundary conditions
\label{h3}
\begin{equation}
\label{BC}
\varphi(y + \pi) \equiv \varphi(y) \qquad \varphi(y) \equiv -\varphi\left(\frac{\pi}{2}-y\right).
\end{equation}
First let us rewrite equation~(\ref{TrigLame}) in the form 
\begin{equation}
\label{SL1}
\frac{d}{dx}\left(\sqrt{1-(k\cos x)^2}\frac{d\varphi}{dx}\right) + \frac{h - 2(k\cos x)^2}{\sqrt{1 - (k\cos x)^2}}\varphi = 0.
\end{equation}

Let us denote $p(x) = \sqrt{1-(k\cos x)^2}$ and an auxiliary Sturm-Liouville problem of the form
\begin{equation}
\label{SL2}
-\frac{d}{dx}\left(p(x)\frac{d\varphi}{dx}\right) + p(x)\varphi = \lambda p(x)\varphi. 
\end{equation}
It easy to see that a function $\varphi(x)$ is a solution of equation~(\ref{SL1}) with $h(k) = 2$ if and only if $\varphi(x)$ is a solution of equation~(\ref{SL2}) with $\lambda(k) = 3$.

Therefore $h_3(k)\ne 2$ iff the Rayleigh quotient
\begin{equation}
\label{R}
R_k[f] = \frac{\displaystyle\int\limits_0^\pi p(k,x)((f')^2+f^2)\,dx}{\displaystyle\int\limits_0^\pi p(k,x)f^2\,dx}
\end{equation}
is greater than $3$ for any functions $f$ satisfying condition~(\ref{BC}). Indeed, by variational characterization of eigenvalues the first eigenvalue $\hat\lambda_0(k)$ of the problem 
~(\ref{SL2}) with boundary conditions~(\ref{BC}) is equal to $\inf R[f]$, where infimum is taken over the subspace $\mathcal L$ of functions $f\in H^1[0,2\pi]$ satisfying conditions~(\ref{BC}).

Then let us remark that the Rayleigh quotient~(\ref{R}) is a decreasing function of $k$. Indeed, if $k_1>k_2$ then $p(k_1,x)<p(k_2,x)$ and we have 
$$
\int\limits_0^\pi p(k_1,x)(f')^2\,dx< \int\limits_0^\pi p(k_2,x)(f')^2\,dx.
$$ 
By adding to both sides $\int\limits_0^\pi p(k_1,x)f^2\,dx\int\limits_0^\pi p(k_2,x)f^2\,dx$, we obtain
$$
\int\limits_0^\pi p(k_1,x))((f')^2+f^2)\int\limits_0^\pi p(k_2,x)f^2\,dx < \int\limits_0^\pi p(k_2,x)((f')^2+f^2)\,dx\int\limits_0^\pi p(k_1,x)f^2\,dx.
$$
This inequality implies $R_{k_1}[f] < R_{k_2}[f]$. 

For $k=1$ the potential $p(k,x)$ becomes $\sin x$ on the segment $[0,\pi]$. Any function $f\in \mathcal L$ can be expressed in the form $g(\cos x)$, where $g\in L^2[-1,1]$ is an odd function and $g'(t)$ is integrable with the weight function $\sqrt{1-t^2}$. Consider the normalized Legendre polynomials $\sqrt\frac{n}{2}P_n(t)$, which form an orthonormal basis in $L^2[-1,1]$. Let us recall that the Legendre polynomials satisfy the Legendre equation,
$$
\frac{d}{dt}\left((1-t^2)\frac{d P_n(t)}{dt}\right) = -n(n+1)P_n(t).
$$ 
Suppose that 
$$
g(t) = \sum\limits_{i=1}^\infty a_n \sqrt\frac{n}{2}P_n(t)
$$ 
is a Fourier expansion for $g(t)$. Then $g'(t)\sqrt{1-t^2}\in L^2[-1,1]$ and the associated Legendre functions $P_n^1(t) = \sqrt{1-t^2}P_n'(t)$ form an orthonormal basis in $L^2$ and let
$$
g'(t)\sqrt{1-t^2} = \sum\limits_{i=1}^\infty b_m \sqrt\frac{m}{2}P_m^1(t).
$$ 
Recall that 
$$
\int\limits_{-1}^1P_n^1(t)P_m^1(t)\,dt = \frac{2n(n+1)}{2n+1}\delta_{m,n}. 
$$
If by $(\cdot,\cdot)$ we denote the $L^2$-inner product in $\mathcal L$, then
\begin{equation*}
\begin{split}
& n(n+1)b_n = \left(g'(t)\sqrt{1-t^2},\sqrt{\frac{m}{2}}P_m^1(t)\right) = -\left(g(t),\sqrt\frac{m}{2}\frac{d}{dt}\left((1-t^2)\frac{dP_n(t)}{dx}\right)\right) =\\
& \sqrt\frac{m}{2}(g(t),n(n+1)P_n(t)) = n(n+1) a_n. 
\end{split}
\end{equation*}
It follows that $a_n  = b_n$.
Now the Raleigh quotient~(\ref{R}) in terms of $g(t)$ has the form 
$$
R_1[g] = \frac{\displaystyle\int\limits_{-1}^1 (1-t^2)g'^2(t) + g^2(t)\,dt}{\displaystyle\int\limits_{-1}^1 g^2(t)\,dt}.
$$
Substituting the series for $g(t)$ and $g'(t)\sqrt{1-t^2}$ into this quotient we see that infimum is reached on $g(t) = P_1(t) = t$ and the quotient is equal to $3$. Thus, $\lambda(k)>3$ for $0<k<1$.

Then it is easy to see that $h_3(0) = 4$ and $h_3(k)$ depend continuosly on $k$. Since $h_3(k)\ne 2$, one has $h_3(k)>2$ for every $k\in (0,1)$. 
\end{proof}

\subsection{Proof of Proposition~\ref{lambda3}.}
Let us first remark that equation~(\ref{assSL}) is the Lam\'e equation with parameters
$$
k^2 = \frac{m^2-n^2}{m^2 + 2mn}, \qquad h = \frac{(m^2 + mn)\lambda - l^2}{m^2+2mn},\qquad n(n+1) = \lambda.
$$
Suppose the contradiction to the statement, i.e. $\lambda_3(0)<2$. Then, since $\lambda_3(n)>\lambda_2(n) = 2$, there exists a number $l_2$ such that $\lambda_3(l_2) = 2$. Then for $l=l_2$ equation~(\ref{assSL}) has a solution with $4$ zeroes on $[0,2\pi)$. Therefore, so does the Lam\'e equation. But such a solution corresponds to either $h_3(k)$ or $h_4(k)$ and one has 
$$
h_4(k)\geqslant h_3(k) \geqslant 2 \quad\mathrm{or}\quad  \frac{(m^2 + mn)\lambda - l-2^2}{m^2+2mn} \geqslant 2,
$$ 
which implies $l_2^2<0$. We obtain a contradiction.
$\qed$

\subsection{Value of the corresponding functional.}
In this section we prove Proposition~\ref{MainProposition}. We start with the formula for the volume of $M_{m,n}$.
\label{proof1}
\begin{equation}
\begin{split}
& \vol(M_{m,n}) = \frac{2\pi}{\sqrt{2}}\int\limits_0^{2\pi}\frac{m^2+2mn+n^2 - (m^2-n^2)\cos 2x}{\sqrt{m^2+4mn+n^2 - (m^2-n^2)\cos 2x}}\,dx = \\
& 8\pi\int\limits_0^{\frac{\pi}{2}}\frac{m^2+mn - (m^2-n^2)\sin^2x}{\sqrt{m^2+2mn - (m^2-n^2)\sin^2 x}}\,dx = \\
& 8\pi\left(\sqrt{m^2+2mn}E\left(\sqrt{\frac{m^2-n^2}{m^2+2mn}}\right) - \frac{mn}{\sqrt{m^2+2mn}}K\left(\sqrt{\frac{m^2-n^2}{m^2+2mn}}\right)\right).
\label{Vol}
\end{split}
\end{equation}
If $mn\equiv 1\mod 2$ then one has to take into account the symmetry $(x,y)\mapsto (x+\pi,y+\pi)$, hence this number has to be divided by $2$.

Now, following~\cite{Karpukhin1}, we prove the non-maximality of the metric on $M_{m,n}$.
Let us recall two propositions from the paper~\cite{Karpukhin1}.
\begin{proposition}
The following inequality holds
$$
\sup \Lambda_n(\mathbb{T}^2,g)> 8\pi n.
$$
\label{LowerBound}
\end{proposition}
\begin{proposition}
For every $k\in[0,1]$ one has
$$
K(k) - \frac{2}{2-k^2}E(k)\geqslant 0.
$$
\label{EKest}
\end{proposition}

By Proposition~\ref{LowerBound} in order to prove the non-maximality of tori $M_{m,n}$ it is sufficient to prove the following proposition.
\begin{proposition} If $mn\equiv 1 \mod 2$ and $m\ne 1$ then the following inequality holds
$$
8\pi(2(m+n)-3) \geqslant \Lambda_{2m+2n-3}(M_{m,n}).
$$
If $mn\equiv 0 \mod 2$ then the following inequality holds
$$
8\pi(4(m+n)-3) \geqslant \Lambda_{4m+4n-3}(M_{m,n}).
$$
\end{proposition}
\begin{proof}
Assume $mn\equiv 1\mod 2$. Then by formula~(\ref{Vol})
\begin{equation}
\begin{split}
&\Lambda_{2m+2n-3}(M_{m,n}) = 2\vol (M_{m,n}) = \\
&8\pi\left(\sqrt{m^2+2mn}E\left(\sqrt{\frac{m^2-n^2}{m^2+2mn}}\right) - \frac{mn}{\sqrt{m^2+2mn}}K\left(\sqrt{\frac{m^2-n^2}{m^2+2mn}}\right)\right).
\end{split}
\label{1}
\end{equation}
Let us apply Proposition~\ref{EKest} with $k = \sqrt{\dfrac{m^2-n^2}{m^2+2mn}}$. Then we have
$$
-\frac{m^2+4mn+n^2}{2m^2+4mn}K(k)\leqslant E(k).
$$
Applying this inequality to formula~(\ref{1}) we have
$$
\Lambda_{2m+2n-3} \leqslant 8\pi\sqrt{m^2+2mn}\left(1 - \frac{2mn}{m^2+4mn+n^2}\right)E(k).
$$
Therefore in order to prove the first inequality it is sufficient to obtain the inequality
\begin{equation}
\sqrt{m^2+2mn}\left(1 - \frac{2mn}{m^2+4mn+n^2}\right)E(k) \leqslant 2m+2n - 3.
\label{2}
\end{equation}
Let us divide both parts of inequality~(\ref{2}) by $m$ and denote the ratio $\dfrac{n}{m}$ by $x\in[0,1]$. Then formula~(\ref{2}) transforms into
$$
\sqrt{1+2x}\left(1 - \frac{2x}{1 + 4x+ x^2}\right) E\left(\sqrt{\frac{1 - x^2}{1+2x}}\right) \leqslant 2(1+x) - \frac{3}{m}.
$$
Since $E(\hat k)\leqslant \dfrac{\pi}{2}$ for each $\hat k\in[0,1]$, this inequality could be obtained from the following one
\begin{equation}
\frac{6}{m} \leqslant 4(1+k) - \pi\sqrt{1 + 2k}.  
\label{3}
\end{equation}
Inequality~(\ref{3}) holds for $m\geqslant 7$. Thus we have several exceptional cases: $\{m,n\} = \{3,1\},\,\{5,1\}\,\{5,3\}\,\{7,1\},\,\{7,3\},\,\{7,5\}$. For these cases
inequality~(\ref{2}) can be verified explicitly using the tables of elliptic integrals in the book~\cite{BF}.

Proof of the second inequality is obtained in the same way. There are also exceptional cases: $\{m,n\} = \{2,1\}, \,\{3,2\}$.   
\end{proof}

\subsection*{Acknowledgements.}
The author thanks A.~V.~Penskoi for fruitful discussions on spectral geometry and the help in the preparation of the manuscript. The author is also grateful to A. E. Mironov for bringing author's attention to the paper~\cite{Mironov}. 

The research of the author was partially supported by Dobrushin Fellowship.

\textsc{{Department of Geometry and Topology, Faculty of Mechanics and Mathematics, Moscow State University, Leninskie Gory, GSP-1, 
119991, Moscow, Russia}}

\smallskip

\textit{{and}}

\smallskip

\textsc{{Independent University of Moscow, Bolshoy Vlasyevskiy pereulok 11, 119002, Moscow, Russia}}

\smallskip

\textit{E-mail address:} \texttt{karpukhin@mccme.ru}
\end{document}